\documentclass{article}

\usepackage[noadjust]{cite}

\usepackage{amsmath,amssymb,amsfonts, amsthm}
\usepackage{nccmath,mathtools}
\usepackage{geometry}

\usepackage{enumitem}
\setlist{topsep=1pt,itemsep=0pt,labelwidth=8pt,labelsep=5pt,leftmargin=16pt}
\usepackage{graphicx}
\usepackage{textcomp}
\usepackage{xcolor}

\usepackage[ruled]{algorithm2e}
\let\oldnl\nl
\newcommand{\nonl}{\renewcommand{\nl}{\let\nl\oldnl}}
\usepackage{setspace}

\usepackage{comment}

\DeclareMathOperator*{\argmin}{arg\,min}

\DeclarePairedDelimiterX\Econd[2]{[}{]}{#1 \mkern2mu\delimsize\vert\mkern2mu #2}

\newtheorem{lemma}{Lemma}
\newtheorem{theorem}{Theorem}
\newtheorem{corollary}{Corollary}

\theoremstyle{definition}
\newtheorem{assumption}{Assumption}

\theoremstyle{remark}
\newtheorem{remark}{Remark}

\newlength\myindent
\newcommand\bindent{%
  \begingroup
  \setlength{\itemindent}{\myindent}
  \addtolength{\algorithmicindent}{\myindent}
}
\newcommand\eindent{\endgroup}

\allowdisplaybreaks

\title{Zeroth-Order Katyusha
}

\title{\LARGE \bf
Zeroth-Order Katyusha: An Accelerated Derivative-Free Method for Composite Convex Optimization
}

\author{Silan Zhang and Yujie Tang
\thanks{This work was supported by the National Natural Science Foundation of China through grant 72301008.
S. Zhang and Y. Tang are with the College of Engineering, Peking University, Beijing, China. 
        Email: {\tt\small zhangsilan@stu.pku.edu.cn}, {\tt\small yujietang@pku.edu.cn}}%
}

\date{}

\begin{document}

\maketitle

\begin{abstract}
We investigate accelerated zeroth-order algorithms for smooth composite convex optimization problems. While for unconstrained optimization, existing methods that merge 2-point zeroth-order gradient estimators with first-order frameworks usually lead to satisfactory performance, for constrained/composite problems, there is still a gap in the complexity bound that is related to the non-vanishing variance of the 2-point gradient estimator near an optimal point. To bridge this gap, we propose the Zeroth-Order Loopless Katyusha (\emph{ZO-L-Katyusha}) algorithm, leveraging the variance reduction as well as acceleration techniques from the first-order loopless Katyusha algorithm. We show that \emph{ZO-L-Katyusha} is able to achieve accelerated linear convergence for compositve smooth and strongly convex problems, and has the same oracle complexity as the unconstrained case. Moreover, the number of function queries to construct a zeroth-order gradient estimator in \emph{ZO-L-Katyusha} can be made to be $O(1)$ on average. These results suggest that \emph{ZO-L-Katyusha} provides a promising approach towards bridging the gap in the complexity bound for zeroth-order composite optimization.
\end{abstract}

\section{Introduction}
\label{section:introduction}

Zeroth-order (ZO) or derivative-free optimization problems have been the subject of intense studies recently. They have emerged from various practical situations where explicit gradient information of the objective function can be very difficult or even impossible to compute, and only function evaluations are accessible to the decision maker. Zeroth-order optimization methods have seen wide applications in multiple disciplines, including signal processing~\cite{liu2020primer}, parameter tuning for machine learning~\cite{snoek_practical_2012}, black-box adversarial attacks on deep neural networks~\cite{chen_zoo_2017}, model-free control~\cite{ren_lqr_2021}, etc.

Various types of zeroth-order algorithms have been proposed to address the challenge of lacking gradient information in the optimization procedure. In this paper, we particularly focus on the line of methods initiated by~\cite{nemirovskij1983problem} and further developed by~\cite{nesterov_random_2017,duchi2015optimal,ghadimi2013stochastic} and other related works. Basically, in this line of methods, a zeroth-order gradient estimator is constructed based on function values acquired at a few randomly explored points, which will then be combined with first-order optimization frameworks, leading to a zeroth-order optimization algorithm. For instance, \cite{nesterov_random_2017} shows that, by combining a two-point Gaussian random gradient estimator with the vanilla gradient descent and Nesterov's accelerated gradient descent iterations, we can obtain a zeroth-order optimization algorithm for deterministic unconstrained smooth convex problems with an oracle complexity of $O(d/\epsilon)$ and $O(d/\sqrt{\epsilon})$, respectively, where $d$ denotes the problem dimension. Similar techniques have also been adapted to the stochastic settings and nonconvex settings~\cite{nesterov_random_2017,duchi2015optimal,shamir2017optimal,ghadimi2013stochastic}. Recently, \cite{balasubramanian2022zeroth} proposed a zeroth-order stochastic approximation algorithms based on the Frank-Wolfe method, focusing particularly on handling constraints,
high dimensionality and saddle-point avoiding. \cite{zhang2022new} proposed a one-point feedback scheme that queries the function value once per iteration but achieves comparable performance with two-point zeroth-order algorithms. \cite{kungurtsev2021zeroth} and \cite{pougkakiotis2023zeroth} studied zeroth-order algorithms for stochastic weakly convex optimization. \cite{pmlr-v202-ren23b} investigated how to escape saddle points for zeroth-order algorithms using two-point gradient estimators.
In sum, for \emph{unconstrained} problems, existing works have shown that we are generally able to design zeroth-order algorithms with oracle complexities at most $O(d)$ times greater than their first-order counterparts.

However, for \emph{constrained} deterministic optimization problems, we notice a gap in terms of oracle complexity between zeroth-order and first order methods. In~\cite{pmlr-v89-sahu19a}, the authors proposed a zeroth-order Frank-Wolf algorithm with an oracle complexity of $O(d/\epsilon)$, but their method employs a gradient estimator whose computation requires $O(d)$ function queries, making it less flexible for problems of high dimensions. The recent work~\cite{liu2024inexact} proposed an inexact preconditioned zeroth-order algorithm for nonconvex composite optimization problems, based on a $2d$-point gradient estimator and also estimation of second-order information using finite differences. Classical two-point gradient estimators, on the other hand, can only achieve a suboptimal oracle complexity $O(d/\epsilon^2)$ for both Frank-Wolf~\cite{balasubramanian2018zeroth} and projected gradient methods~\cite{duchi2015optimal}. It has been suspected that the major cause of this gap is the non-vanishing variance of two-point gradient estimators as the iterate approaches an optimal point on the boundary, which could slow down convergence and lead to higher complexities~\cite{jin2023zeroth}. Moreover, zeroth-order methods that apply acceleration techniques to constrained problems are scanty, and to the best of our knowledge, none of them has achieved a theoretical oracle complexity that can match the first-order counterpart.

The non-vanishing variance of the gradient estimator suggests the adaption of variance reduction as a possible approach to close the aforementioned gap in zeroth-order optimization. Variance reduction techniques have been widely used in machine learning problems to further enhance the performance of stochastic-gradient-descent-type methods. Algorithms such as SAG\cite{schmidt2017minimizing}, SAGA\cite{defazio2014saga}, SVRG\cite{johnson2013accelerating}, SARAH\cite{nguyen2017sarah} and SPIDER\cite{fang2018spider} have emerged to solve finite-sum problems that are prevalent in the machine learning community. Along this line, some variants incorporating first-order acceleration techniques have also been proposed, including ASVRG~\cite{shang2018asvrg}, Katyusha~\cite{allen2018katyusha} and its loopless counterpart L-Katyusha~\cite{kovalev2020don}. We noticed that some existing works~\cite{fang2018spider,ji2019improved,huang2019faster} have investigated combining variance reduction techniques with zeroth-order optimization. However, \cite{fang2018spider,ji2019improved} considers unconstrained nonconvex problems, and \cite{huang2019faster} considers composite nonconvex problems. Moreover, these works concentrate more on the finite-sum problems or stochastic problems for machine learning tasks. 
The power of variance reduction techniques on reducing the variance of zeroth-order gradient estimators in deterministic constrained settings is still underexplored.

\subsection{Our Contributions}
In this paper, we develop an accelerated proximal zeroth-order algorithm for composite optimization problems with strongly convex objectives. The proposed algorithm leverages both variance reduction and acceleration techniques, and achieves superior convergence rate compared to existing zeroth-order algorithms. Specifically, our contributions can be summarized as follows:
\begin{itemize}
    \item We propose \emph{ZO-L-Katyusha}, an accelerated proximal zeroth-order algorithm, to minimize a strongly convex composite objective function. One advantage of \emph{ZO-L-Katyusha} is that, with properly chosen algorithmic parameters, the algorithm only needs $O(1)$ function queries per iteration on average. Together with the loopless structure, this makes it more flexible than algorithms based on $O(d)$-point finite difference gradient estimators, especially when the dimension of the problem $d$ is high. On the other hand, \emph{ZO-L-Katyusha} also allows using purely $O(d)$-point gradient estimators, making the algorithm flexible and applicable to a wide range of practical scenarios.

    \item We prove that the oracle complexity of \emph{ZO-L-Katyusha} can be upper bounded by $O(d\sqrt{\kappa}\ln (1/\epsilon))$, where $\kappa$ is the reciprocal condition number of the objective function. This is significantly better than the state-of-art zeroth-order algorithms for \emph{constrained} or \emph{composite} strongly convex optimization problems. To the best of our knowledge, this is the first accelerated zeroth-order algorithm that can match the performance of first order algorithms up to a factor of variable dimension under constrained settings.
    
    
\end{itemize}

\subsection{Notations}

Throughout this paper, we denote the standard inner product in $\mathbb{R}^n$ by $\langle x,y\rangle= x^T y$, and denote the Euclidean norm by $\|x\| = \sqrt{x^T x}$. We use $e_i$ to denote the vector that has only one non-zero entry $1$ at its $i$-th coordinate, and use $I_d$ to denote the $d\times d$ identity matrix. The unit sphere in $\mathbb{R}^d$ will be denoted by $\mathbb{S}_d\coloneqq \{x\in\mathbb{R}^d: \|x\|=1\}$. For any set $\mathcal{S}$, $|\mathcal{S}|$ denotes its cardinality.

\section{Problem Formulation}
In this paper, we consider finding a solution to the following optimization problem:
\begin{equation}\label{eq:formulation}
    \min_{x\in \mathbb{R}^d}\ F(x)\stackrel{\mathrm{def}}{=}f(x)+\psi(x).
\end{equation}
Here $ f:\mathbb{R}^d \rightarrow \mathbb{R}$ is a convex and differentiable function, and $\psi : \mathbb{R}^d \rightarrow \mathbb{R} \cup \{+\infty \}$ is convex and lower semicontinuous but possibly non-differentiable. We impose the following assumptions on the information that can be accessed by the decision maker:
\begin{itemize}
\item The decision maker may only obtain the value of the function $f$ by querying a black-box, or a zeroth-order oracle, which returns the value $f(x)$ whenever a point $x\in\mathbb{R}^d$ is fed into the black-box as its input. The decision maker does not have access to the derivatives of any order of $f$.
\item The decision maker is able to efficiently evaluate the following proximal operator
\[
\operatorname{prox}_{\eta\psi}(x) \coloneqq
\argmin_{y\in\mathbb{R}^d}\left\{
\frac{1}{2}\|x-y\|^2 + \eta\psi(y)
\right\}
\]
for any $\eta>0$.
\end{itemize}
The first assumption, as we have shown in Section~\ref{section:introduction}, is relevant to many problems in areas such as signal processing and machine learning, and necessitates the development of a zeroth-order optimization algorithm. The second assumption is standard in composite optimization and has been imposed in relevant studies such as~\cite{huang2019faster,pougkakiotis2023zeroth}, and allows us to incorporate simple nonsmoothness into our problem settings and algorithm design. For example, if we take $\psi$ to the indicator function of a closed and convex set $\mathcal{X}\subseteq\mathbb{R}^d$:
\[
\psi(x) = I_{\mathcal{X}}(x) = \begin{cases}
0, & \text{if}\ x\in\mathcal{X}, \\
+\infty, & \text{otherwise},
\end{cases}
\]
Then the main problem~\eqref{eq:formulation} is equivalent to the constrained optimization problem
$
\min_{x\in\mathcal{X}} f(x)
$,
and the proximal operator reduces to the projection operator onto the convex set $\mathcal{X}$. We see that our problem formulation includes constrained optimization as a special case.

We also make the following technical assumptions that will facilitate convergence analysis of our algorithm:
\begin{assumption}
\begin{enumerate}
\item The function $f$ is $L$-smooth on $\mathbb{R}^d$ for some $L>0$, i.e.,
\[
\|\nabla f(x)-\nabla f(y)\|\leq L\|x-y\|,\qquad\forall x,y\in\mathbb{R}^d.
\]
\item There exist some $\mu_f \geq 0$ and $\mu_\psi\geq 0$ such that for all $x, y \in R^d$,
\begin{align*}
f(x) & \geq f(y)+\langle\nabla f(y),x-y\rangle+\frac{\mu_f}{2}\|x-y\|^2, \\
\psi(x) & \geq \psi(y)+\langle v,x-y\rangle+\frac{\mu_{\psi}}{2}\|x-y\|^2,
\end{align*}
where $v$ is any subgradient of $\psi$ at $y$.

\item $\mu \coloneqq \mu_f+\mu_\psi>0$, so that $F$ is $\mu$-strongly convex.
\end{enumerate}
\end{assumption}

The strongly convex setting will be the main focus of this paper, and investigation of the non-strongly convex setting has been our ongoing work. But we expect that the algorithm design techniques (presented in Section~\ref{section:algorithm_design}) can also be adapted to non-strongly convex problems.

\subsection{Gaps in the Algorithm Design}
\label{subsection:gaps}

When $\psi=0$, i.e., the main problem~\eqref{eq:formulation} is an unconstrained smooth convex problem, we can find zeroth-order algorithms in existing literature that can solve~\eqref{eq:formulation} with satisfactory performance. For example, \cite{nesterov_random_2017,duchi2015optimal,shamir2017optimal} investigated the following 2-point zeroth-order gradient estimator:
\begin{equation}\label{eq:two_point_estimator}
\mathsf{G}^{(2)}_f(x;u,\beta)
=d\cdot \frac{f(x+\beta u)-f(x)}{\beta} u,
\end{equation}
where $\beta>0$ is a parameter called the \emph{smoothing radius}, and $u$ is a \emph{random perturbation direction} either sampled from the normal distribution $\mathcal{N}(0,d^{-1}I_d)$ or the uniform distribution on the unit sphere $\mathcal{U}(\mathbb{S}_d)$. It has been shown that for unconstrained problems, by combining the gradient estimator~\eqref{eq:two_point_estimator} with certain first-order frameworks, we can obtain zeroth-order algorithms with convergence guarantees that are similar to the first-order counterparts. Particularly, \cite{nesterov_random_2017} showed that, if we replace the true gradient in Nesterov's accelerated gradient descent by the 2-point gradient estimator~\eqref{eq:two_point_estimator}, as long as the algorithmic parameters are adjusted accordingly, the resulting algorithm can successfully solve an unconstrained $L$-smooth and $\mu$-strongly convex problem with oracle complexity upper bounded by
\[
O\!\left(
d\sqrt{\frac{L}{\mu}}\ln\frac{1}{\epsilon}
\right),
\]
which agrees with the first-order Nesterov's accelerated gradient descent method except for an $O(d)$ factor.

On the other hand, if we shift our focus to constrained or composite convex problems, we soon find that straightforward combination of the 2-point gradient estimator $\mathsf{G}^{(2)}_f(x;u,\beta)$ with a first-order optimization scheme will result in slower convergence and higher complexity than in the unconstrained setting; see Section~\ref{subsection:naive_combination} for a numerical example. The recent study~\cite{jin2023zeroth} suggests that slower convergence of such algorithms is related to the fact that
\[
\mathbb{E}\!\left[
\big\|\mathsf{G}^{(2)}_f(x;u,\beta)-\nabla f(x)\big\|^2
\right]
\sim d\|\nabla f(x)\|^2
\]
assuming $\beta>0$ sufficiently small; here the expectation is taken with respect to $u\sim\mathcal{N}(0,I_d)$ or $u\sim\mathcal{U}(\mathbb{S}_d)$. For unconstrained smooth problems, as the algorithm's iterate $x^k$ approaches an optimal point, we have $\nabla f(x^k)\rightarrow 0$, and consequently $\mathbb{E}\!\left[
\big\|\mathsf{G}^{(2)}_f(x^k;u,\beta)-\nabla f(x^k)\big\|^2
\right]\approx 0$, indicating that the randomness of $\mathsf{G}^{(2)}_f(x^k;u,\beta)$ will have a diminishing effect and thus the algorithm's behavior resembles a deterministic first-order method. However, for constrained or composite problems, the gradient at an optimal point may not vanish. As a result, as the iterate $x^k$ approaches an optimal point, the quantity $\mathbb{E}\!\left[
\big\|\mathsf{G}^{(2)}_f(x^k;u,\beta)-\nabla f(x^k)\big\|^2
\right]$ in general will not be close to zero, indicating that the algorithm's behavior resembles a stochastic first-order algorithm, which has strictly inferior convergence behavior.

There are also existing works (including \cite{pmlr-v89-sahu19a,ji2019improved,liu2024inexact}, etc.) that propose the following $(d+1)$-point gradient estimator (or its variants) for solving constrained convex problems:
\begin{equation}\label{eq:2d_point_estimator}
\sum_{i=1}^d
\frac{f(x+\beta e_i)-f(x)}{2\beta}e_i.
\end{equation}
However, each computation of the gradient estimator~\eqref{eq:2d_point_estimator} then requires $O(d)$ evaluations of the function $f$, which may not be desired when the problem dimension $d$ is high, as argued by~\cite{pmlr-v202-ren23b}. Besides, for constrained composite convex optimization problems, convergence results for accelerated zeroth-order algorithm using $O(d)$-point gradient estimators are yet to be established.

The above discussion demonstrates that there still exists a gap between constrained/composite zeroth-order methods and unconstrained zeroth-order methods, especially if we insist on employing gradient estimators that can be constructed by only $O(1)$ function queries. On the other hand, the above analysis also suggests that the variance of the zeroth-order gradient estimator plays a key role in the algorithm's convergence behavior, which inspires us that we may consider incorporating variance reduction techniques into zeroth-order optimization, to bridge the aforementioned gap. Later, we shall see that, our proposed algorithm provides a promising approach towards bridging this gap: It achieves accelerated linear convergence for composite smooth and strongly convex problems, while the construction of each gradient estimator requires $O(1)$ function queries \emph{on average}.

\section{Our Algorithm}
\label{section:algorithm_design}

In this section, we present our proposed algorithm, the Zeroth-Order loopless Katyusha (\textit{ZO-L-Katyusha}) algorithm, while is presented in Algorithm~\ref{alg:zol_katyusha}.

\begin{algorithm}[!tbh]
\DontPrintSemicolon
\SetAlgoLined
\LinesNumbered
\SetKwInput{KwParam}{Parameters}
\SetKwInput{KwInit}{Initialization}
\KwParam{Stepsize parameters $\theta \in (0, 1)$, $M>0$; batch size $|\mathcal{S}| \leq d$; smoothing radius $\beta>0$; probability $p \in (0, 1]$}
\KwInit{$\sigma=\mu_f/M$, $\eta = 1/(3\theta)$, $y^0=z^0=w^0\in\mathbb{R}^d$ }
\For{$k=0,1,2,\ldots$}{
$\displaystyle x^k = \theta z^k + \mfrac{1}{2} w^k + (\mfrac{1}{2}-\theta)y^k$\;
\vspace{3pt}

Generate a set of random vectors $\mathcal{S}_k$ by
\vspace{3pt}

\nonl\begin{minipage}{13.5cm}
\begin{itemize}
\item \textbf{Option I}: Drawing $|\mathcal{S}|$ samples uniformly from $\{e_1, e_2, ..., e_d\}$ without replacement

\item \textbf{Option II}: Drawing $|\mathcal{S}|$ samples independently from $\mathcal{U}(\mathbb{S}_d)$, the uniform distribution over the unit sphere
\end{itemize}
\end{minipage}
\vspace{3pt}

\nl $g^k \!=\! \hat{\nabla}_{\mathcal{S}_k}f(x^k) - 
\mfrac{1}{|\mathcal{S}_k|}
\!\sum\limits_{u\in\mathcal{S}_k}
\!\!\big\langle \hat{\nabla} f(w^k),u\big\rangle u
+ \hat{\nabla}\!f(w^k)$\;
\vspace{3pt}

$\displaystyle z^{k+1}=\operatorname{prox}_{\frac{\eta}{(1+\eta\sigma)M}\psi}\!\bigg(\frac{\eta\sigma x^{k}+z^{k}-\frac{\eta}{M}g^{k}}{1+\eta\sigma}\bigg)$\;
\vspace{3pt}

$y^{k+1} = x^k + \theta (z^{k+1} - z^k)$\;

$w^{k+1} = 
            \begin{cases}
                y_k, & \text{with probability } p\\
                w_k, & \text{with probability } 1-p
            \end{cases}$\;
}
\caption{ZO-L-Katyusha}\label{alg:zol_katyusha}
\end{algorithm}

The algorithm generally follows the framework of loopless Katyusha~\cite{kovalev2020don}: We keep a reference point $w^k$ at which we evaluate a relatively accurate estimation of the gradient $\hat{\nabla} f(w^k)$, and use $\hat{\nabla} f(w^k)$ to adjust the subsequent estimated gradients. The reference point $w^k$ will be updated in a random fashion as proposed by~\cite{kovalev2020don}. However, the major difference is that we use variance reduction for handling the randomness of the two-point zeroth-order gradient estimator, rather than deadling with the finite-sum problem structure. Specifically, we set the adjusted gradient estimator to be
\begin{equation}\label{eq:g_k_def}
g^k = \hat{\nabla}_{\mathcal{S}_k}f(x^k) - 
\frac{1}{|\mathcal{S}_k|}
\sum_{u\in\mathcal{S}_k}
d\left\langle \hat{\nabla} f(w^k), u\right\rangle u
+ \hat{\nabla}f(w^k),
\end{equation}
where we denote
\begin{align}
\hat{\nabla}_{\mathcal{S}}f(x) 
& =
\frac{1}{|\mathcal{S}|}\sum_{u \in \mathcal{S}} d\cdot\frac{f(x+\beta u) - f(x)}{\beta}u, \label{eq:2_point_algorithm} \\
\hat{\nabla}f(x) & =
\sum_{i=1}^d \frac{f(x+\beta e_i) - f(x)}{\beta }e_i.
\end{align}
In~\eqref{eq:g_k_def}, each $\mathcal{S}_k$ is a randomly generated set of identically distributed perturbation directions for $|\mathcal{S}_k|$ 2-point gradient estimators, which will be averaged to form $\hat{\nabla}_{\mathcal{S}_k} f(x^k)$ as shown in~\eqref{eq:2_point_algorithm}. For simplicity, we employ the same batch size (which we denote by $|\mathcal{S}|$) for every $\mathcal{S}_k$. Unlike variance reduction algorithms for finite-sum problems, here we choose to draw $|\mathcal{S}_k|$ random directions, rather than drawing individual loss functions. The reference gradient $\hat{\nabla} f(w^k)$, on the other hand, is generated by the $(d+1)$-point gradient estimator, giving sufficiently accurate approximation to $\nabla f(w^k)$ when the smoothing radius $\beta$ is chosen properly.

We offer two sampling strategies for the set $\mathcal{S}_k$, as presented in Algorithm~\ref{alg:zol_katyusha}: The first option is to draw $S$ vectors from $\{e_1, e_2,\ldots, e_d\}$ \emph{without replacement}, i.e., we choose the $|\mathcal{S}|$ random perturbation directions to be distinct and parallel to the coordinate axes; as an extreme case, when $|\mathcal{S}|=d$, we get a $(d+1)$-point gradient estimator. The second option is to draw $|\mathcal{S}|$ vectors independently from $\mathcal{U}(\mathbb{S}_d)$, the uniform distribution on the unit sphere. We shall see that these two sampling strategies lead to largely similar performance guarantees when $|\mathcal{S}|\ll d$.

Algorithm~\ref{alg:zol_katyusha} is a relatively general zeroth-order algorithm with several tunable parameters. In Section~\ref{section:convergence_analysis}, we will show how these parameters should generally be chosen to guarantee convergence. Here, we discuss some special choices of the parameters:
\begin{itemize}
\item Suppose we set $|\mathcal{S}|=1$ and $p=(\gamma d)^{-1}$ for some $\gamma>0$, then on average we update $w^k$ after approximately $\gamma d$ iterations. Since each reference gradient $\hat{\nabla } f(w^k)$ requires $d+1$ function queries, while each $\hat{\nabla}_{\mathcal{S}_k} f(x^k)$ requires $2$ function queries, we see that the averaged number of function queries needed per zeroth-order gradient estimation is
\[
\frac{(d+1)\times 1 + 2\times (\gamma d)}{1+\gamma d} = O(1).
\]
We shall see in Section~\ref{section:convergence_analysis} that such choice of $|\mathcal{S}|$ and $p$ guarantees convergence as long as other parameters are selected accordingly. This justifies our previous claim that, in \emph{ZO-L-Katyusha}, the number of function queries for the construction of each gradient estimator can be made to be $O(1)$ on average.

\item If we use the first sampling strategy, and set the mini-batch size $|\mathcal{S}| = d$ and probability $p=1$, then \textit{ZO-L-Katyusha} reduces to an accelerated zeroth-order algorithm using purely $(d+1)$-point gradient estimators, with $w^{k+1}=y^k$ and $g^k = \hat{\nabla}f(x^k)$.
We can see that \emph{ZO-L-Katyusha} provides abundant flexibility which allows the algorithm to be applicable to a wide range of practical scenarios.
\end{itemize}

\section{Convergence Analysis}
\label{section:convergence_analysis}

In this section, we present our theoretical results on the convergence rate of the proposed \emph{ZO-L-Katyusha} algorithm. We first present a main theorem that covers general sampling strategies and parameter choices. Then we provide corollaries on the algorithm's oracle complexities for some special cases. The proof of the main theorem will be postponed to Section~\ref{subsection:proof_outline}.

\begin{theorem}
\label{theorem:main}
Let $x^\ast\in\mathbb{R}^d$ be the unique optimal solution to~\eqref{eq:formulation}. Denote
\[
A = \begin{cases}
\displaystyle\max\left\{\frac{4d(d-|\mathcal{S}|)}{(d-1)|\mathcal{S}|},1\right\}, & \text{Option I is used},  \vspace{6pt}
\\
\displaystyle\frac{4 d}{|\mathcal{S}|}, &
\text{Option II is used},
\end{cases}
\]
and suppose in \emph{ZO-L-Katyusha}, we set
$M = (A+1)L/3$. Denote
\[
\begin{aligned}
\Psi^k \coloneqq{} &
\frac{\mu+3\theta M}{2}\|z^k-x^\ast\|^2
+\frac{1}{\theta}(F(y^k)-F(x^\ast))
+
\frac{1+\theta}{2p\theta}(F(w^k)-F(x^\ast)).
\end{aligned}
\]
Then for the corresponding sequence generated by ZO-L-Katyusha, we have
\[
\mathbb{E}\!\left[\Psi^k\right]
\leq 
(1-\Delta)^k \Psi^0 + \frac{\beta^2 d^2 L}{\Delta}
\left(\frac{L}{d\mu}+\frac{1}{A\theta}\right),
\]
where
\[
\Delta = 
\min\left\{
\frac{\mu}{2\mu+6\theta M},
\frac{\theta}{2},
\frac{p\theta}{1+\theta}
\right\}.
\]
\end{theorem}

The convergence guarantees presented in Theorem~\ref{theorem:main} is quite general, and applies to arbitrary choices of $\theta\in(0,1)$, $|\mathcal{S}|\leq d$, $\beta>0$ and $p\in(0,1]$. Next we provide corollaries for three special cases, with $\theta$ and $q$ fixed so that we can get concrete results on the complexity of our algorithm.

\begin{corollary}[Mini-batch, Option I]\label{coro:mini_batch}
Suppose $\mathcal{S}_k$ is drawn uniformly from $\{e_1, e_2, \ldots, e_d\}$ without replacement, with $|\mathcal{S}| \leq \sqrt{d}$. Given arbitrary $\epsilon>0$, let the parameters of \emph{ZO-L-Katyusha} satisfy
\[
\begin{aligned}
M & = \frac{4d(d-|\mathcal{S}|)L}{3(d-1)|\mathcal{S}|}+\frac{L}{3},
 & \quad
\theta & = \min\left\{\sqrt{\frac{d\mu}{M}}, \frac{1}{2}\right\}, &\quad
\beta & =
O\!\left(\sqrt{\frac{\mu \epsilon}{d^{\frac{3}{2}}L^2}} \right),
&\quad p&=\frac{1}{d}.
\end{aligned}
\]
Then we achieve $\mathbb{E}\!\left[F(w^k) - F(x^\ast)\right] \leq \epsilon$ for
\[
k \geq O\!\left(d\sqrt{\frac{L}{\mu|\mathcal{S}|}}\,\ln\frac{1}{\epsilon}\right).
\]
\end{corollary}
\begin{corollary}[Mini-batch, Option II]\label{coro:mini_batch_2}
Suppose each vector in $\mathcal{S}_k$ is drawn independently from $\mathcal{U}(\mathbb{S}_d)$, the uniform distribution on the unit sphere, with $|\mathcal{S}| \leq \sqrt{d}$. Given arbitrary $\epsilon>0$, let the parameters of \emph{ZO-L-Katyusha} satisfy
\begin{align*}
M & = \frac{4dL}{|\mathcal{S}|}+\frac{L}{3},
 & \quad
\theta & = \min\left\{\sqrt{\frac{d\mu}{M}}, \frac{1}{2}\right\}, &\quad
\beta & =
O\!\left(\sqrt{\frac{\mu \epsilon}{d^{\frac{3}{2}}L^2}} \right),
&\quad
p& =\frac{1}{d}.
\end{align*}
Then we achieve $\mathbb{E}\!\left[F(w^k) - F(x^\ast)\right] \leq \epsilon$ for
\[
k \geq O\!\left(d\sqrt{\frac{L}{\mu|\mathcal{S}|}}\,\ln\frac{1}{\epsilon}\right).
\]
\end{corollary}
\begin{corollary}[Full batch, Option I]\label{coro:full_batch}
Suppose we choose $\mathcal{S}_k=\{e_1,e_2,\ldots,e_d\}$ for every $k\geq 0$. Given arbitrary $\epsilon>0$, let the parameters of \emph{ZO-L-Katyusha} satisfy
\[
\begin{aligned}
M & = \frac{2L}{3},
 & \quad
\theta & = \min\left\{\sqrt{\frac{\mu}{M}}, \frac{1}{2}\right\}, &\quad
\beta & =
O\!\left(\sqrt{\frac{\mu \epsilon}{d^2 L^2}} \right),
&\quad  p&=1.
\end{aligned}
\]
Then we achieve $\mathbb{E}\!\left[F(w^k) - F(x^\ast)\right] \leq \epsilon$ for
\[
k \geq O\!\left(\sqrt{\frac{L}{\mu}}\,\ln\frac{1}{\epsilon}\right).
\]
\end{corollary}

We provide further discussions on these results:
\begin{enumerate}
\item In the mini-batch setups (analyzed in Corollaries~\ref{coro:mini_batch} and~\ref{coro:mini_batch_2}), if we choose $|\mathcal{S}| = 1$, then as discussed at the end of Section~\ref{section:algorithm_design}, the averaged number of zeroth-order oracle calls (i.e., function queries) for each calculation of a gradient estimator is $O(1)$. Thus, the oracle complexity of ZO-L-Katyusha can be upper bounded by
$O\!\left(d\sqrt{L/\mu}\ln (1/\epsilon)\right)$.

\item In the full-batch setup with Option I (analyzed in Corollary~\ref{coro:full_batch}), since the number of zeroth-order oracle calls for each gradient estimator is $O(d)$, the oracle complexity of ZO-L-Katyusha is also upper bounded by $O\!\left(d\sqrt{L/\mu}\ln (1/\epsilon)\right)$, the same as the mini-batch setups with $|\mathcal{S}|=1$.
\end{enumerate}
In sum, regardless of whether we choose the mini-batch setup or the full-batch setup, the algorithm enjoys the same oracle complexity bound. Since first-order Nesterov's accelerated gradient descent has an oracle complexity upper bounded by $O\left(\sqrt{L/\mu}\ln (1/\epsilon)\right)$ for smooth and strongly convex optimization, we see that our algorithm's complexity bound accords with that of first-order accelerated methods apart from an $O(d)$ factor. In other words, our algorithm provides a promising approach towards bridging the gap mentioned and discussed in Section~\ref{subsection:gaps}.

\subsection{Proof Outline}
\label{subsection:proof_outline}

Here we provide the outline of the proof of Theorem~\ref{theorem:main}. Due to space limitations, we postpone the proofs of some technical lemmas to the Appendix. Throughout the proof, we let $\mathcal{F}_k$ denote the filtration generated by the family of random vectors $\{z^t,w^t,y^t:\,t\leq k\}$.

We first present some lemmas that characterize the bias and mean square error of the adjusted gradient estimator with respect to the true gradient.

\begin{lemma}
\label{lemma:Option_I_bias_var}
Suppose Option I is adopted as the sampling strategy for $\mathcal{S}_k$. Then
\begin{equation}
\|\mathbb{E}\Econd{g^k}{\mathcal{F}_k}-\nabla f(x^k)\|^2\leq \frac{1}{4}L^2\beta^2d,
\end{equation}
and
\begin{align}
\mathbb{E}\Econd{\|g^k-\nabla f(x^{k})\|^{2}}{\mathcal{F}_k}
\leq{} &\frac{4d(d-|\mathcal{S}|)L}{(d-1)|\mathcal{S}|}\cdot(f({w^k})-f(x^{k})-\langle\nabla f(x^k),w^k-x^{k}\rangle)
+ 2L^2 \beta^2 d^2.\label{eq:variance1}
\end{align}
\end{lemma}
\begin{lemma}
\label{lemma:Option_II_bias_var}
Suppose Option II is adopted as the sampling strategy for $\mathcal{S}_k$. Then
\begin{equation}
\|\mathbb{E}\Econd{g^k}{\mathcal{F}_k}-\nabla f(x^k)\|^2\leq L^2\beta^2,
\end{equation}
and
\begin{align}\label{eq:variance2}
\mathbb{E}\Econd{\|g^k-\nabla f(x^{k})\|^{2}}{\mathcal{F}_k}
\leq{} & \frac{4dL}{|\mathcal{S}|}\cdot(f({w^k})-f(x^{k})-\langle\nabla f(x^k),w^k-x^{k}\rangle)
+ 2L^2 \beta^2 d^2.
\end{align}
\end{lemma}

The next lemma provides a necessary technical result for establishing convergence.
\begin{lemma}
We have
    \begin{align}\label{eq:lemma3}
		&\quad\frac{1}{\theta}\big(f(y^{k+1})-f(x^k)\big)- \frac{\eta }{2(M-L\eta\theta)}\|g^k - \nabla f(x^k)\|^2\nonumber\\
  &\leq\frac{M}{2\eta}\|z^{k+1}-z^k\|^2+\langle g^k,z^{k+1}-z^k\rangle.
    \end{align}
\end{lemma}

Now, for notational simplicity, we introduce the following quantities:
\[
    \begin{aligned}
	\mathcal{Z}^{k} & =\frac{M+\eta\mu}{2\eta}\|z^{k}-x^{*}\|^{2}, &\ \ &
 \mathcal{Y}^{k}=\frac{1}{\theta}(F(y^{k})-F^{*}),\\ 
 \mathcal{W}^{k} & =\frac{1+\theta}{2p\theta}(F(w^{k})-F^{*}),
    \end{aligned}
\]
so that the stochastic Lyapunov function $\Phi^k$ can be written as $\Phi^k = \mathcal{Z}^k + \mathcal{Y}^k + \mathcal{W}^k$.
Since $w^{k+1}$ equals $w^k$ with probability $1-p$ and equals $y^k$ with probability $p$, it's not hard to check that the following identity holds:
\begin{equation}
\label{eq:conditinoal_expectation_w_k+1}
	\mathbb{E}\Econd{\mathcal{W}^{k+1}}{\mathcal{F}_k}=(1-p)\mathcal{W}^{k}+\frac{1+\theta}{2}\mathcal{Y}^{k},\qquad \forall k\geq0.
\end{equation}
Using this set of notations, we proceed to establish the following lemma, which is another necessary technical result for our proof.
\begin{lemma}
We have
    \begin{align}\label{eq:lemma4}
	&\quad \langle g^{k},x^{*}-z^{k+1}\rangle+\frac{\mu_{f}}{2}\|x^{k}-x^{*}\|^{2}+\frac{M\mathcal{Z}^{k}}{M+\eta\mu}\nonumber\\
        &\geq\frac{M}{2\eta}\|z^{k}-z^{k+1}\|^{2}+\mathcal{Z}^{k+1}+\psi(z^{k+1})-\psi(x^{*}).
    \end{align}
\end{lemma}

Now we are ready to prove our main theorem. First of all, we use the $\mu_f$-strong convexity of $f$ to obtain
\begin{align*}
f(x^{*})\geq{} &
f(x^{k})+\langle\nabla f(x^{k}),x^{*}-x^{k}\rangle+\frac{\mu_{f}}2\|x^{k}-x^{*}\|^{2}\\
	={} & 
 f(x^{k})+\frac{\mu_{f}}{2}\|x^{k}-x^{*}\|^{2}
 +\langle\nabla f(x^{k}),x^{*}-z^{k}\rangle
 +\langle\nabla f(x^{k}),z^{k}-x^{k}\rangle\\
	={} &
 f(x^{k})+\frac{\mu_{f}}{2}\|x^{k}-x^{*}\|^{2}+\langle\nabla f(x^{k}),x^{*}-z^{k}\rangle\\
        &\!\!+\frac{1}{2\theta}\langle\nabla\! f(x^{k}),x^{k} \!-\! w^{k}\rangle
        +\frac{1 \!-\! 2\theta}{2\theta}\langle\nabla\! f(x^{k}),x^{k} \!-\! y^{k}\rangle \\
\geq{} &
f(x^{k})+\frac{\mu_{f}}{2}\|x^{k}-x^{*}\|^{2}+\langle\nabla f(x^{k}),x^{*}-z^{k}\rangle\\
        &+\frac{1}{2\theta}\langle\nabla f(x^{k}),x^{k} \!-\! w^{k}\rangle
        +\frac{1\!-\!2\theta}{2\theta}(f(x^k) \!-\! f(y^k)),
\end{align*}
where in the third step we used Line 2 in Algorithm~\ref{alg:zol_katyusha}, and in the last step we used the convexity of $f$ to get $\langle\nabla f(x^{k}),x^{k}-y^{k}\rangle\geq f(x^k)-f(y^k)$. Then, we notice that
\begin{align*}
\langle\nabla f(x^{k}),x^{*}-z^{k}\rangle
={} &
-\langle \mathbb{E}\Econd{g^k}{\mathcal{F}_k}-\nabla f(x^k),x^\ast-z^k\rangle
+
\langle \mathbb{E}\Econd{g^k}{\mathcal{F}_k},x^\ast-z^{k}\rangle,
\end{align*}
and using the inequality $\langle a,b\rangle
\geq -\frac{1}{2\epsilon}\|a\|^2-\frac{\epsilon}{2}\|b\|^2$ for any $\epsilon>0$, we have
\begin{align*}
\langle\mathbb{E}\Econd{
g^k}{\mathcal{F}_k} - \nabla f(x^k),
x^\ast-z^k
\rangle
\geq{} &
-\frac{1}{\mu}
\left\|\mathbb{E}\Econd{
g^k}{\mathcal{F}_k} - \nabla f(x^k)\right\|^2
-\frac{\mu}{4}
\left\|x^\ast-z^k\right\|^2 \\
\geq{} & -\frac{\beta^2 L^2d}{\mu}
-\frac{\eta\mu}{2(M+\eta\mu)}\mathcal{Z}^k,
\end{align*}
where we used Lemmas~\ref{lemma:Option_I_bias_var} and~\ref{lemma:Option_II_bias_var} to bound the term $\left\|\mathbb{E}\Econd{
g^k}{\mathcal{F}_k} - \nabla f(x^k)\right\|^2$, and also used the definition of $\mathcal{Z}^k$.
Therefore
\begin{align*}
f(x^\ast)\geq{} &
f(x^k) + \frac{1}{2\theta}\langle\nabla f(x^k),x^k-w^k\rangle
-\frac{\eta\mu}{2(M+\eta\mu)}\mathcal{Z}^k \\
& -\frac{\beta^2 L^2 d}{\mu}
+\frac{1-2\theta}{2\theta}(f(x^k)-f(y^k))
+ \frac{\mu_f}{2}\|x^k-x^\ast\|^2 + \langle \mathbb{E}\Econd{g^k}{\mathcal{F}_k}, x^\ast-z^{k}\rangle.
\end{align*}
Next, we note that the inequality~\eqref{eq:lemma4} implies
\[
\begin{aligned}
& \frac{\mu_f}{2}\|x^k-x^\ast\|^2
+\langle \mathbb{E}\Econd{g^k}{\mathcal{F}_k}, x^\ast-z^{k}\rangle \\
={} & \mathbb{E}\Econd*{
\langle g^k, x^\ast-z^{k+1}\rangle
+\frac{\mu_f}{2}\|x^k-x^\ast\|^2
}{\mathcal{F}_k}
+ \mathbb{E}\Econd*{\langle g^k, z^{k+1}-z^k\rangle}{\mathcal{F}_k} \\
\geq{} &
\mathbb{E}\Econd*{\frac{M}{2\eta}\|z^k-z^{k+1}\|^2
+\mathcal{Z}^{k+1} + \psi(z^{k+1})-\psi(x^\ast)}{\mathcal{F}_k}
+ \mathbb{E}\Econd*{
\langle g^k,z^{k+1}-z^k\rangle
-\frac{M\mathcal{Z}^k}{M+\eta\mu} 
}{\mathcal{F}_k},
\end{aligned}
\]
and \eqref{eq:lemma3} together with Lemmas~\ref{lemma:Option_I_bias_var} and~\ref{lemma:Option_II_bias_var} leads to
\begin{align*}
& \mathbb{E}\Econd*{\frac{M}{2\eta}\|z^k-z^{k+1}\|^2
+\langle g^k,z^{k+1}-z^k\rangle}
{\mathcal{F}_k} \\
\geq{} &
\mathbb{E}\Econd*{
\frac{f(y^{k+1}) \!-\! f(x^k)}{\theta}
-\frac{\eta}{2(M\!-\!L\eta\theta)}\|g^k \!-\! \nabla f(x^k)\|^2
}{\mathcal{F}_k} \\
\geq{} &
\frac{1}{\theta}\mathbb{E}\Econd{f(y^{k+1})-f(x^k)}{\mathcal{F}_k}
-\frac{\eta}{2(M-L\eta\theta)}\cdot 2L^2\beta^2d^2 \\
& -
\frac{\eta AL}{2(M \!-\! L\eta\theta)}
(f(w^k)-f(x^k)-\langle\nabla f(x^k),w^k-x^k\rangle) \\
={} & \frac{1}{\theta}\mathbb{E}\Econd{f(y^{k+1})-f(x^k)}{\mathcal{F}_k}
- \frac{L\beta^2 d^2}{A\theta}
-\frac{1}{2\theta}
(f(w^k)-f(x^k)-\langle\nabla f(x^k),w^k-x^k\rangle),
\end{align*}
where in the last step we used $\eta=1/(3\theta)$ and $M=(A+1)L/3$. Consequently,
\begin{align*}
& \frac{\mu_f}{2}\|x^k-x^\ast\|^2
+\langle \mathbb{E}\Econd{g^k}{\mathcal{F}_k}, x^\ast-z^{k}\rangle \\
\geq{} &
\frac{1}{\theta}\mathbb{E}\Econd{f(y^{k+1})-f(x^k)}{\mathcal{F}_k}
- \frac{L\beta^2 d^2}{A\theta}
-\frac{1}{2\theta}
(f(w^k)-f(x^k)-\langle\nabla f(x^k),w^k-x^k\rangle) \\
& +
\mathbb{E}\Econd*{\mathcal{Z}^{k+1} + \psi(z^{k+1})-\psi(x^\ast)
-\frac{M\mathcal{Z}^k}{M+\eta\mu} }{\mathcal{F}_k}.
\end{align*}
Summarizing the previous results, we can show that
\begin{align*}
& \mathbb{E}\Econd*{f(x^\ast)-\psi(z^{k+1}) + \psi(x^\ast) -\mathcal{Z}^{k+1}}{\mathcal{F}_k} \\
\geq{} &
f(x^k) +
\frac{1-2\theta}{2\theta}(f(x^k)-f(y^k))
-\frac{M\mathcal{Z}^k}{M+\eta\mu}
+\frac{1}{\theta}\mathbb{E}\Econd{f(y^{k+1})-f(x^k)}{\mathcal{F}_k} \\
& -\frac{1}{2\theta}(f(w^k)-f(x^k))
-\frac{\beta^2L^2d}{\mu}
-\frac{\eta\mu}{2(M+\eta\mu)}\mathcal{Z}^k
-\frac{L\beta^2d^2}{A\theta} \\
={} &
-\frac{M+\eta\mu/2}{M+\eta\mu}\mathcal{Z}^k
-\frac{1-2\theta}{2\theta}f(y^k)
+\frac{1}{\theta}\mathbb{E}\Econd{f(y^{k+1})}{\mathcal{F}_k}
-\frac{1}{2\theta} f(w^k)
-\frac{\beta^2L^2d}{\mu} - \frac{L\beta^2 d^2}{A\theta}.
\end{align*}
Moreover, since $\psi$ is convex and 
\begin{align*}
y^{k+1} & =x^k+\theta (z^{k+1}-z^k)
=\theta z^{k+1}+\frac{1}{2} w^k+\left(\frac{1}{2}-\theta\right)y^k,
\end{align*}
by Jensen's inequality, we have
\begin{equation*}
\psi(z^{k+1})\ge\frac{1}{\theta}\psi(y^{k+1})-\frac{1}{2\theta}\psi(w^{k})-\frac{1-2\theta}{2\theta}\psi(y^{k}).
\end{equation*}
Hence, we arrive at
\begin{align*}
f(x^{*})\geq{} &
\mathbb{E}\Econd{\mathcal{Z}^{k+1}}{\mathcal{F}_k}-\frac{M+\eta\mu/2}{M+\eta\mu}\mathcal{Z}^{k}
-\frac{1-2\theta}{2\theta}F(y^{k}) \\
& +\frac{1}{\theta}\mathbb{E}\Econd{F(y^{k+1})}{\mathcal{F}_k}-\frac{F(w^{k})}{2\theta}-\psi(x^{*})-C(\beta),
\end{align*}
where we denote
\[
C(\beta) \coloneqq \beta^2 d^2 L \left(
\mfrac{L}{d\mu} + \mfrac{1}{A\theta}\right).
\]
After arranging the terms using~\eqref{eq:conditinoal_expectation_w_k+1}, we get
\begin{align*}
\mathbb{E}\Econd*{\mathcal{Z}^{k+1}+\mathcal{Y}^{k+1}+\mathcal{W}^{k+1}}{\mathcal{F}_k}
\leq{} &
\frac{M+\eta\mu/2}{M+\eta\mu}\mathcal{Z}^{k}
+\left(\frac{1}{2}-\theta\right)\mathcal{Y}^{k}+\frac{p}{1+\theta}\mathcal{W}^{k} \\
&+(1-p)\mathcal{W}^{k} +\frac{1+\theta}{2}\mathcal{Y}^{k}+C(\beta) \\
={} &
\left(1-\frac{\eta\mu/2}{M+\eta\mu}\right)\mathcal{Z}^{k}
+\left(1-\frac{\theta}{2}\right)\mathcal{Y}^{k}\\
&+\left(1-\frac{p\theta}{1+\theta}\right)\mathcal{W}^{k}+C(\beta) \\
\leq{} & (1-\Delta)(\mathcal{Z}^k + \mathcal{Y}^k + \mathcal{W}^k) + C(\beta),
\end{align*}
where we used the definition of the quantity $\Delta$ in the last step. In other words,
\[
\mathbb{E}\Econd*{\Phi^{k+1}
-\frac{1}{\Delta}C(\beta)}{\mathcal{F}_k}
\leq (1-\Delta)\left(\Phi^{k}
-\frac{1}{\Delta}C(\beta)\right).
\]
The conclusion of the theorem is now evident.

\section{Numerical experiments}
\label{section:numerical}
\subsection{Algorithmic Performance} \label{section:numerical1}
In this section, we consider the following constrained optimization problem:
\begin{equation*}
    \min_{x\in\mathcal{X}}\left\{\frac{1}{n}\sum_{i=1}^{n}\ln(1+\exp(-b_i a_i ^T x)) + \frac{\mu}{2}\|x\|^2\right\},
\end{equation*}
which takes the form of an L2 regularized logistic regression problem with constraints. Here we let $\mathcal{X}\subseteq\mathbb{R}^d$ be a closed box, and set $d=40, n=30, \mu=0.02$ for the test case. The data set $\{(a_i,b_i)\}_{i=1}^n$ is manually synthesized.
This problem can be formulated as a special case of \eqref{eq:formulation} where
\begin{equation*}
    f(x) = \frac{1}{n}\sum_{i=1}^{n}\log(1+\exp(-b_i a_i ^T x)),\quad 
\end{equation*}
is smooth and convex, and
\begin{equation*}
        \quad \psi(x) = \frac{\mu}{2}\|x\|^2 + I_{\mathcal{X}}(x)
\end{equation*}
is strongly convex. Note that here we only use this problem for the purpose of comparison of the algorithms' convergence behavior; a more practical approach would also need to exploit its finite-sum structure.

\begin{figure}[tbp]
    \centering
    \includegraphics[width=.6\linewidth]{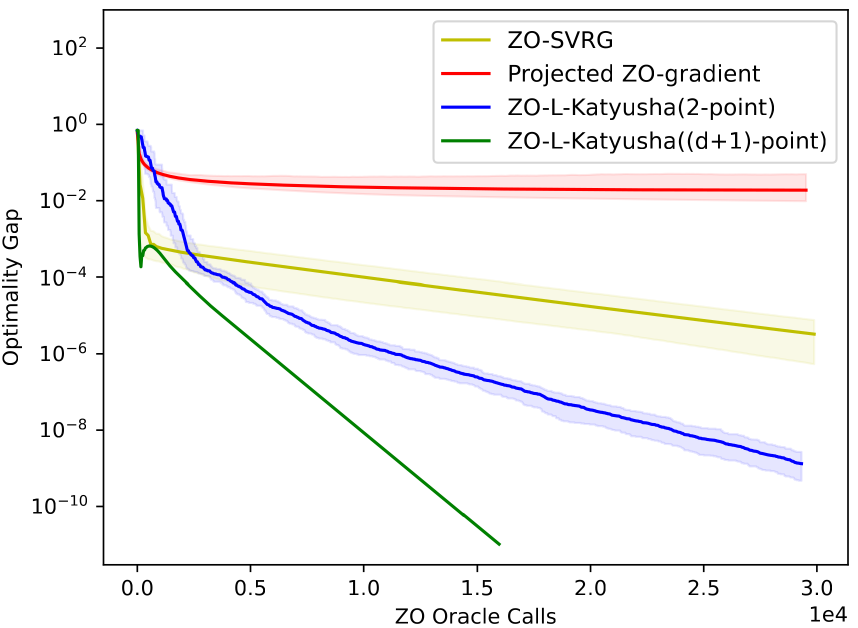}
    \caption{The performance of $2$-point and $(d+1)$-point ZO-Katyusha}
    \label{fig:1}
\end{figure}

We test the $2$-point and $(d+1)$-point version of \textit{ZO-L-Katyusha} (i.e., the mini-batch setup and the full-batch setup) against existing algorithms. The results are shown in Figure~\ref{fig:1}. Here \textit{ZO-SVRG} is adapted from \cite{ji2019improved}, with some minor adaptations to fit our specific problem formulation, and \textit{Projected ZO-gradient} is the standard stochastic projected zeroth-order gradient descent algorithm, to be found in \cite{duchi2015optimal}. Each solid curve in the Figure represents the average of 50 random trials for each algorithm, and the corresponding light shade represents the 5\% to 95\% quantile interval among these trials.

From Figure~\ref{fig:1}, some observations can be made. First, \textit{Projected ZO-gradient} performs the worst, with the curve almost flattened out in the latter half of the iterations, this may be due to the diminishing stepsize policy introduced to deal with the non-vanishing variance of the gradient estimator, which leads to a sub-optimal $O(d/\epsilon^2)$ function query complexity. While with variance reduction techniques, all the other algorithms can employ a constant stepsize. Second, our \textit{ZO-Katyusha} algorithm performs better than \textit{ZO-SVRG}, while these two algorithms use similar gradient estimators, acceleration enables the former to converge even faster.

\subsection{A Note on Accelerated Zeroth-Order Method}
\label{subsection:naive_combination}
It is already mentioned that simply plugging a 2-point ZO gradient estimator into existing first order projected accelerated methods, like the one in \cite{nesterov2018lectures}, does not result in fast convergence as one may expect. To demonstrate this, we test the performance of the following iterations:
\begin{equation}
\label{eq:naive_combination}
\begin{aligned}
    x^{k+1} &= P_\mathcal{X}\!\left(y_k - \mfrac{1}{dL}g^k\right),\\
    y^{k+1} &= x^{k+1} - \mfrac{1-\sqrt{\mu/L}}{1+\sqrt{\mu/L}}(x^{k+1}-x^k),
\end{aligned}    
\end{equation}
where $P_\mathcal{X}$ denotes projection onto $\mathcal{X}$, and $g^k$ is generated by the simple 2-point gradient estimator~\eqref{eq:2d_point_estimator}. For the numerical test, we set $f$ to be a strongly convex quadratic function, and set $\mathcal{X}$ to be i) $\mathbb{R}^d$ (the unconstrained case), and ii) a high-dimensional box (the constrained case), respectively. We run the above algorithm 50 times over the test cases.

\begin{figure}[tbp]
    \centering
    \includegraphics[width=.6\linewidth]{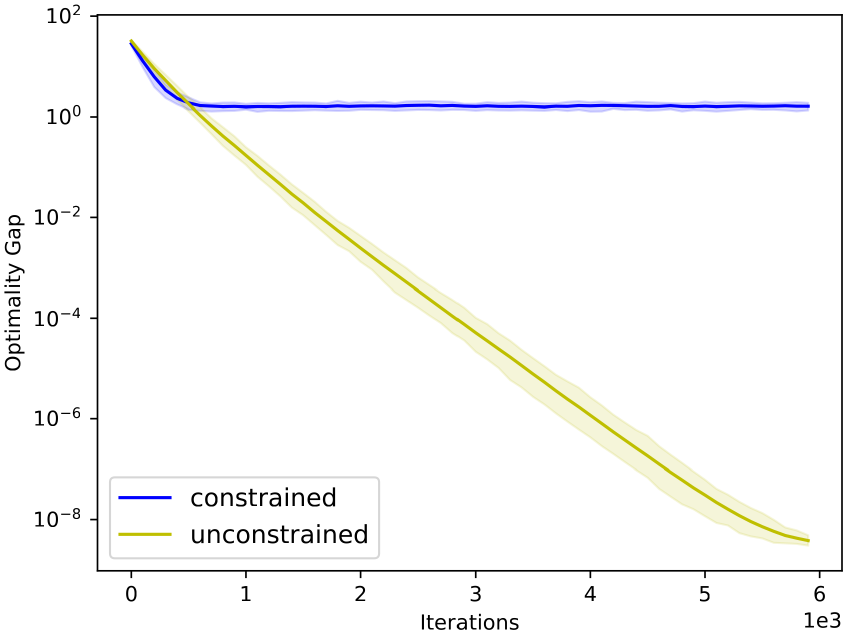}
    \caption{Performance of the algorithm~\eqref{eq:naive_combination} under constrained and unconstrained cases.}
    \label{fig:2}
\end{figure}

The results are shown in Figure~\ref{fig:2}. We see that while in the unconstrained case, this algorithm converges quite fast, in the constrained case, the algorithm does not exhibit a clear trend of convergence. This partially illustrates the gap mentioned in Section~\ref{subsection:gaps}, and suggests that some variance reduction techniques may be necessary if we want to design accelerated zeroth-order algorithms based on $2$-point gradient estimators.

\section{Conclusion}

We proposed the Zeroth-Order Loopless Katyusha (\emph{ZO-L-Katyusha}) algorithm for composite strongly convex optimization, which is able to achieve accelerated linear convergence while only requiring $O(1)$ function queries per iteration. Some future directions include: i) analysis of zeroth-order Katyusha methods for non-strongly convex problems; ii) developing accelerated zeroth-order algorithms for constrained problems using purely $O(1)$-point gradient estimators; iii) extending our algorithm to incorporate finite-sum structures.

\bibliographystyle{IEEEtran}
\bibliography{IEEEabrv,Katyusha}

\appendix
	\section{Technical Proofs}
Throughout this material, we agree on the following notation conventions:
\[
\begin{aligned}
	\hat{\nabla}_u f(x) & = d\frac{f(x+\beta u) - f(x)}{\beta} u,
 &\quad
	\nabla_u f(x) & = d\langle \nabla f(x), u \rangle u,
\end{aligned}
\]
for $u \in \mathbb{R}^d$, and
\begin{equation*}
	\nabla_{\mathcal{S}} f(x) = \frac{1}{|\mathcal{S}|} \sum_{u \in \mathcal{S}} \nabla_u f(x)
\end{equation*}
for $\mathcal{S} = \{u_1, u_2, ..., u_{|\mathcal{S}|}\}, u_i \in \mathbb{R}^d$. With this notation, we have
\[
\hat{\nabla} f(x)
=\frac{1}{d}\sum_{i=1}^d
\hat{\nabla}_{e_i} f(x),
\qquad
\nabla f(x)
=\frac{1}{d}\sum_{i=1}^d
\nabla_{e_i} f(x).
\]
Recall that the function $f$ is assumed to be convex and $L$-smooth, and $\beta>0$ represents the smoothing radius.

Note that
\[
\begin{aligned}
\mathbb{E}\Econd*{g^k}{\mathcal{F}_k}
={} & 
\mathbb{E}\Econd*{
\hat\nabla_{\mathcal{S}_k} f(x^k)
-\frac{d}{|\mathcal{S}_k|}\sum_{u\in\mathcal{S}_k}
uu^T\hat\nabla f(w^k) + \hat\nabla f(w^k),
}{\mathcal{F}_k}
\end{aligned}
\]
and since $
\mathbb{E}\Econd*{
\sum_{u\in\mathcal{S}_k} uu^T
}{\mathcal{F}_k}
=\frac{|\mathcal{S}_k|}{d}I_d$ holds for both Option I and Option II in Algorithm~\ref{alg:zol_katyusha} (the proof for Option II can be found in, e.g., Lemma 5 of \cite{ji2019improved}),
we see that
\[
\begin{aligned}
\mathbb{E}\Econd*{g^k}{\mathcal{F}_k}
&
=\mathbb{E}\Econd*{\hat{\nabla}_{\mathcal{S}_k} f(x^k)}{\mathcal{F}_k}.
\end{aligned}
\]

To proceed, we introduce the following auxiliary lemmas:
\setcounter{lemma}{4}
\setcounter{equation}{14}
\begin{lemma}
Suppose $u\in\mathbb{S}_d$. Then for any $x\in \mathbb{R}^d$, we have
	\begin{equation}\label{eq:two_point_bias}
		\|\hat{\nabla}_u f(x) - \nabla_u f(x)\| \leq \frac{ d L\beta}{2}.
	\end{equation}
\end{lemma}
\begin{proof} We have
	\begin{align*}
		\|\hat{\nabla}_u f(x) - \nabla_u f(x)\| &= d\left| \frac{f(x+\beta u) - f(x)}{\beta} - \langle \nabla f(x), u \rangle \right|\\
		&=\frac{d}{\beta}\left|\int_0 ^{\beta} \left(\langle\nabla f(x+\gamma u), u\rangle - \langle\nabla f(x), u\rangle\right) d\gamma \right|\\
		&\leq\frac{d}{\beta}\int_0 ^{\beta} \left|\langle\nabla f(x+\gamma u), u\rangle - \langle\nabla f(x), u\rangle\right| d\gamma \\
		&\leq \frac{d}{\beta}\int_0 ^{\beta}  \|\nabla f(x+\gamma u)-\nabla f(x)\|\|u\|d\gamma\\
		&\leq \frac{d}{\beta} \int_{0}^{\beta} L\gamma \|u\|^2 d\gamma = \frac{dL\beta}{2},
	\end{align*}
 where in the second step we used the fundamental theorem of calculus, and in the fifth step we used the $L$-smoothness of $f$.
\end{proof}

\begin{lemma}[{\cite[Lemma 14]{JMLR:v21:19-198}}]
\label{lemma:bias_2point_uniform_sphere}
Suppose $u\sim \mathcal{U}(\mathbb{S}_d)$. Then for any $x\in\mathbb{R}^d$, we have
\[
\left\|\mathbb{E}\!\left[
\hat\nabla_u f(x)
\right] - \nabla f(x)\right\|
\leq \beta L.
\]
\end{lemma}

\subsection{Proof of Lemma 1}
Note that for Option I, we have $\mathbb{E}\Econd*{\hat{\nabla}_{\mathcal{S}_k} f(x^k)}{\mathcal{F}_k}=\hat{\nabla} f(x^k)$, and therefore
\[
\begin{aligned}
	\left\|\mathbb{E}\Econd*{g^k}{\mathcal{F}_k}-\nabla f(x^k)\right\|^2 &= \left\| \hat{\nabla}f(x^k) - \nabla f(x^k)\right\|^2\\
	&= \frac{1}{d^2}\sum_{i=1}^{d} \left\|\hat{\nabla}_{e_i} f(x^k) - \nabla_{e_i} f(x^k)\right\|^2
 \stackrel{\eqref{eq:two_point_bias}}{\leq} \frac{d L^2 \beta^2}{4}.
\end{aligned}
\]
Then, for the proof of \eqref{eq:variance1}, we have
\begin{align}
& \left\|g^k - \nabla f(x^k)\right\|^2 \nonumber\\
={} &
\Bigg\|
\hat{\nabla}_{\mathcal{S}_k} f(x^k)
-\nabla_{\mathcal{S}_k} f(x^k)
-\left(\frac{d}{|\mathcal{S}_k|}
\sum_{u\in\mathcal{S}_k} uu^T - I_d \right)
\left(\hat{\nabla} f(w^k) - \nabla f(w^k)\right) \nonumber\\
& 
\qquad +\nabla_{\mathcal{S}_k} f(x^k)
-\left(\frac{d}{|\mathcal{S}_k|}
\sum_{u\in\mathcal{S}_k} uu^T - I_d \right)
\nabla f(w^k)
-\nabla f(x^k)
\Bigg\|^2 \nonumber\\
\leq{} &
4 \left\|\hat\nabla_{\mathcal{S}_k} f(x^k)
-\nabla_{\mathcal{S}_k} f(x^k)\right\|^2
+ 4
\left\|
\left(
\frac{d}{|\mathcal{S}_k|}
\sum_{u\in\mathcal{S}_k} uu^T - I_d
\right)
\left(
\hat\nabla f(w^k) - \nabla f(w^k)
\right)\right\|^2 \nonumber\\
& + 2
\left\|
\nabla_{\mathcal{S}_k} f(x^k)
-\nabla_{\mathcal{S}_k} f(w^k)
+\nabla f(w^k)
-\nabla f(x^k)
\right\|^2.
\label{eq:gk_gradient_diff_step1}
\end{align}
For the first term on the right-hand side, notice that all elements in $\mathcal{S}_k$ are orthogonal for Option~I, and so
\[
\begin{aligned}
\left\|
\hat\nabla_{\mathcal{S}_k} f(x^k) - \nabla_{\mathcal{S}_k} f(x^k)
\right\|^2
={} &
\sum_{u\in\mathcal{S}_k}
\left\|
\frac{1}{|\mathcal{S}_k|}\left(
\hat\nabla_u f(x^k) - \nabla_u f(x^k)\right)
\right\|^2
\leq
\frac{\beta^2 L^2 d^2}{4|\mathcal{S}_k|}.
\end{aligned}
\]
Then for the second term, we notice that
\[
\begin{aligned}
& \mathbb{E}\Econd*{\left\|
\left(
\frac{d}{|\mathcal{S}_k|}
\sum_{u\in\mathcal{S}_k} uu^T - I_d
\right)
\left(
\hat\nabla f(w^k) - \nabla f(w^k)
\right)\right\|^2}{\mathcal{F}_k} \\
={} &
\left(
\hat\nabla f(w^k) - \nabla f(w^k)
\right)^{\!T}
\mathbb{E}\Econd*{
\left(
\frac{d}{|\mathcal{S}_k|}
\sum_{u\in\mathcal{S}_k} uu^T - I_d
\right)^2
}{\mathcal{F}_k}
\left(
\hat\nabla f(w^k) - \nabla f(w^k)
\right) \\
={} &
\left(
\hat\nabla f(w^k) - \nabla f(w^k)
\right)^{\!T}
\mathbb{E}\Econd*{
\left(\frac{d^2}{|\mathcal{S}_k|^2}
-2\frac{d}{|\mathcal{S}_k|}\right)
\sum_{u\in\mathcal{S}_k} uu^T
+ I_d
}{\mathcal{F}_k}
\left(
\hat\nabla f(w^k) - \nabla f(w^k)
\right) \\
={} & 
\frac{d-|\mathcal{S}_k|}{|\mathcal{S}_k|}
\left\|\hat\nabla f(w^k) - \nabla f(w^k)\right\|^2
\leq
\frac{d(d-|\mathcal{S}_k|)L^2\beta^2}{4|\mathcal{S}_k|},
\end{aligned}
\]
where in the second step we used $
\left(\sum_{u\in\mathcal{S}_k} uu^T\right)^2
=\sum_{u\in\mathcal{S}_k} uu^T$, and in the last step we used
\[
\left\|\hat\nabla f(w^k) - \nabla f(w^k)\right\|^2
=\frac{1}{d^2}
\sum_{i=1}^d \left\|\hat{\nabla}_{e_i} f(w^k) - \nabla_{e_i} f(w^k)\right\|^2
 \stackrel{\eqref{eq:two_point_bias}}{\leq} \frac{d L^2 \beta^2}{4}.
\]
As a result, we get
\[
\begin{aligned}
&\mathbb{E}\Econd*{\left\|g^k-\nabla f(x^{k})\right\|^{2}}{\mathcal{F}_k} \\
\leq{} &
2\,
\mathbb{E}\Econd*{\left\|
\nabla_{\mathcal{S}_k} f(x^k)
-\nabla_{\mathcal{S}_k} f(w^k)
+\nabla f(w^k)
-\nabla f(x^k)
\right\|^2}{\mathcal{F}_k}
+
\frac{2d^2-|\mathcal{S}_k|d}{|\mathcal{S}_k|}\beta^2 L^2.
\end{aligned}
\]
Now to simplify notation, given $\mathcal{S}_k $,  we let
\begin{align*}
	b_i & = \nabla_{e_i} f(x^k) -  \nabla_{e_i } f(w^k)  - (\nabla f(x^k) - \nabla f(w^k)),
 &
	\iota_i & = \begin{cases}
		1, & \text{if} \ e_i  \in \mathcal{S}_k,\\
		0, & \text{otherwise}.
	\end{cases}
\end{align*}
Then it is obvious that $\mathbb{E}\Econd*{\iota_i ^2}{\mathcal{F}_k} = \frac{|\mathcal{S}_k|}{d}$ and $\mathbb{E}\Econd*{\iota_i \iota_j}{\mathcal{F}_k} = \frac{|\mathcal{S}_k|(|\mathcal{S}_k|-1)}{d(d-1)}$ for $i\neq j$. Therefore
\begin{align*}
	&\mathbb{E}\Econd*{\left\|\nabla_{\mathcal{S}_k} f(x^k) - \nabla_{\mathcal{S}_k} f(w^k) + \nabla f (w^k) - \nabla f(x^k)\right\|^2}{\mathcal{F}_k}\nonumber\\
	={} &\frac{1}{|\mathcal{S}_k| ^2}\mathbb{E}\Econd*{\left\|\sum_{i=1} ^d \iota_i b_i\right\|^2}{\mathcal{F}_k}\nonumber\\
	={} &\frac{1}{|\mathcal{S}_k| ^2} \left(\sum_{i=1} ^d \mathbb{E}\Econd*{\iota_i ^2 \left \|b_i\right \|^2}{\mathcal{F}_k} + \sum_{i\neq j}\mathbb{E}\Econd*{\iota_i \iota_j \langle b_i, b_j \rangle}{\mathcal{F}_k} \right)\nonumber\\
	={} &\frac{1}{|\mathcal{S}_k|^2}\left(\left(\frac{|\mathcal{S}_k|}{d}-\frac{|\mathcal{S}_k|(|\mathcal{S}_k|-1)}{d(d-1)}\right)\sum_{i=1}^d\left\|b_i\right\|^2+\frac{|\mathcal{S}_k|(|\mathcal{S}_k|-1)}{d(d-1)}\left\|\sum_{i=1}^d b_i\right\|^2\right).\nonumber\\
 ={} &
 \frac{d-|\mathcal{S}_k|}{|\mathcal{S}_k|d(d-1)}
 \sum_{i=1}^d \|b_i\|^2,
\end{align*}
where we used $\sum_{i=1} ^d b_i = 0$. Then since $\sum_{i=1}^{d} \left\|a_i - \frac{1}{d}\sum_{j=1}^{d} a_j\right\|^2 \leq \sum_{i=1}^{d} \|a_i \|^2$, we have
\[
\begin{aligned}
\sum_{i=1}^d \|b_i\|^2
& = \sum_{i=1}^d
\left\|
\nabla_{e_i} f(x^k) - \nabla_{e_i} f(w^k)
-\frac{1}{d}\sum_{j=1}^d
\left(\nabla_{e_j} f(x^k) - \nabla_{e_j} f(w^k)\right)
\right\|^2 \\
& \leq \sum_{i=1}^d \left\|\nabla_{e_i} f(x^k)
-\nabla_{e_i} f(w^k)\right\|^2
= d^2 \|\nabla f(x^k)-\nabla f(w^k)\|^2 \\
& \leq 2L d^2
(f(w^k)-f(x^k)+\langle\nabla f(x^k),w^k-x^k\rangle),
\end{aligned}
\]
where the last step follows from the $L$-smoothness of $f$. Summarizing these results, we get
\[
\begin{aligned}
 & \mathbb{E}\Econd*{
\|g^k-\nabla f(x^k)\|^2
}{\mathcal{F}_k} \\
\leq{} &
\frac{4d(d-|\mathcal{S}|)L}{(d-1)|\mathcal{S}|}
(f(w^k)-f(x^k)+\langle\nabla f(x^k),w^k-x^k\rangle)
+
\frac{2d^2-|\mathcal{S}|d}{|\mathcal{S}|}\beta^2 L^2,
\end{aligned}
\]
and by $\frac{2d^2-|\mathcal{S}|d}{|\mathcal{S}|}
\leq 2d^2$ we conclude the proof.

\subsection{Proof of Lemma 2}

In this part, we shall denote $\mathcal{S}_k = \{u_1^k,\ldots,u_{|\mathcal{S}|}^k\}$.

First, note that for Option II, $u_1^k,\ldots,u_{|\mathcal{S}|}^k$ are i.i.d. and follows the distribution $\mathcal{U}(\mathbb{S}_d)$, we have
\[
\begin{aligned}
\left\|\mathbb{E}\Econd*{g^k}{\mathcal{F}_k} - \nabla f(x^k)\right\|^2
={} & \left\|
\mathbb{E}\Econd*{\hat\nabla_{u_1^k}f(x^k)}{\mathcal{F}_k}-\nabla f(x)\right\|^2
\leq \beta^2 L^2,
\end{aligned}
\]
where the last step follows from Lemma~\ref{lemma:bias_2point_uniform_sphere}.

Then, to prove \eqref{eq:variance2}, we notice that the bound~\eqref{eq:gk_gradient_diff_step1} still holds for Option~II:
\[
\begin{aligned}
& \left\|g^k - \nabla f(x^k)\right\|^2 \\
\leq{} &
4 \left\|\hat\nabla_{\mathcal{S}_k} f(x^k)
-\nabla_{\mathcal{S}_k} f(x^k)\right\|^2
+ 4
\left\|
\left(
\frac{d}{|\mathcal{S}_k|}
\sum_{u\in\mathcal{S}_k} uu^T - I_d
\right)
\left(
\hat\nabla f(w^k) - \nabla f(w^k)
\right)\right\|^2 \\
& + 2
\left\|
\nabla_{\mathcal{S}_k} f(x^k)
-\nabla_{\mathcal{S}_k} f(w^k)
+\nabla f(w^k)
-\nabla f(x^k)
\right\|^2.
\end{aligned}
\]
This time, for the first term on the right-hand side, we have
\[
\begin{aligned}
\mathbb{E}\Econd*{\left\|\hat\nabla_{\mathcal{S}_k} f(x^k)
-\nabla_{\mathcal{S}_k} f(x^k)\right\|^2}{\mathcal{F}_k}
={} &
\mathbb{E}\Econd*{\left\|
\frac{1}{|\mathcal{S}_k|}\sum_i
\left(\hat\nabla_{u_i^k} f(x^k)
-\nabla_{u_i^k} f(x^k)\right)\right\|^2}{\mathcal{F}_k} \\
\leq{} &
\mathbb{E}\Econd*{\left|\frac{1}{|\mathcal{S}_k|}
\sum_i\left\|
\hat\nabla_{u_i^k} f(x^k)
-\nabla_{u_i^k} f(x^k)\right\|\right|^2}{\mathcal{F}_k} \\
\leq{} &
\frac{\beta^2 L^2 d^2}{4},
\end{aligned}
\]
where the last inequality follows from the bound~\eqref{eq:two_point_bias}; in order to bound the second term, we notice that
\[
\begin{aligned}
& \mathbb{E}\Econd*{
\left(
\frac{d}{|\mathcal{S}_k|}
\sum_{u\in\mathcal{S}_k} uu^T - I_d
\right)^2
}{\mathcal{F}_k} \\
={} &
\mathbb{E}\Econd*{
\frac{d^2}{|\mathcal{S}_k|^2}
\sum_{i,j}
u_i^k(u_i^k)^T u_j^k(u_j^k)^T
-2\frac{d}{|\mathcal{S}_k|}
\sum_{i} u_i^k(u_i^k)^T
+ I_d
}{\mathcal{F}_k} \\
={} &
\mathbb{E}\Econd*{
\frac{d^2}{|\mathcal{S}_k|^2}
\left(\sum_{i}
u_i^k(u_i^k)^T u_i^k(u_i^k)^T
+
\sum_{i\neq j}
u_i^k(u_i^k)^T u_j^k(u_j^k)^T
\right)
-2\frac{d}{|\mathcal{S}_k|}
\sum_{i} u_i^k(u_i^k)^T
+ I_d
}{\mathcal{F}_k} \\
={} &
\frac{d^2}{|\mathcal{S}_k|^2}
\left(\sum_i \mathbb{E}\Econd*{u_i^k(u_i^k)^T}{\mathcal{F}_k}
+\sum_{i\neq j} \mathbb{E}\Econd*{u_i^k(u_i^k)^T}{\mathcal{F}_k}
\cdot \mathbb{E}\Econd*{u_j^k(u_j^k)^T}{\mathcal{F}_k}\right) \\
&
-2\frac{d}{|\mathcal{S}_k|}\sum_i \mathbb{E}\Econd*{u_i^k(u_i^k)^T}{\mathcal{F}_k}
+I_d \\
={} & 
\frac{d-1}{|\mathcal{S}_k|} I_d,
\end{aligned}
\]
where we used $\mathbb{E}\Econd*{u_i(u_i^k)^T}{\mathcal{F}_k}=\frac{1}{d}I_d$, the proof of which can be seen in Lemma 5 of \cite{ji2019improved}. Therefore
\[
\begin{aligned}
& \mathbb{E}\Econd*{\left\|
\left(
\frac{d}{|\mathcal{S}_k|}
\sum_{u\in\mathcal{S}_k} uu^T - I_d
\right)
\left(
\hat\nabla f(w^k) - \nabla f(w^k)
\right)\right\|^2}{\mathcal{F}_k} \\
={} &
\left(
\hat\nabla f(w^k) - \nabla f(w^k)
\right)^{\!T}
\mathbb{E}\Econd*{
\left(
\frac{d}{|\mathcal{S}_k|}
\sum_{u\in\mathcal{S}_k} uu^T - I_d
\right)^2
}{\mathcal{F}_k}
\left(
\hat\nabla f(w^k) - \nabla f(w^k)
\right) \\
={} & 
\frac{d-1}{|\mathcal{S}_k|}
\left\|
\hat\nabla f(w^k) - \nabla f(w^k)
\right\|^2 \\
={} & \frac{(d-1)}{|\mathcal{S}_k|}
\frac{1}{d^2}\sum_{i=1}^d
\|\hat\nabla_{e_i}f(w^k) - \nabla e_i f(w^k)\|^2 \\
\leq{} &
\frac{d(d-1)}{|\mathcal{S}_k|}\beta^2 L^2.
\end{aligned}
\]
Consequently,
\begin{equation}
\label{eq:variance2.1}
\begin{aligned}
& \mathbb{E}\Econd*{\left\|g^k - \nabla f(x^k)\right\|^2}{\mathcal{F}_k} \\
\leq{} &
2
\mathbb{E}\Econd*{
\left\|
\nabla_{\mathcal{S}_k} f(x^k)
-\nabla_{\mathcal{S}_k} f(w^k)
+\nabla f(w^k)
-\nabla f(x^k)
\right\|^2}{\mathcal{F}_k}
+ \left(\frac{d(d-1)}{|\mathcal{S}_k|}+d^2\right)L^2\beta^2
\end{aligned}
\end{equation}
Next, because $u_i ^k, i=1, 2, \ldots, |\mathcal{S}|$ are i.i.d. and sampled from $\mathcal{U}(\mathbb{S}_d)$, it follows that
\[
\begin{aligned}
	& \mathbb{E}\Econd*{\left\|\nabla_{\mathcal{S}_k} f(x^k) - \nabla_{\mathcal{S}_k} f(w^k) + \nabla f (w^k) - \nabla f(x^k)\right\|^2}{\mathcal{F}_k}\\
	={} &\frac{1}{| \mathcal{S}_k|}\,\mathbb{E}\Econd*{\left\|\nabla_{u_1 ^k} f(x^k) - \nabla_{u_1 ^k} f(w^k) + \nabla f (w^k) - \nabla f(x^k)\right\|^2}{\mathcal{F}_k} \\
	\leq{} & \frac{1}{|\mathcal{S}_k|}\,\mathbb{E}\Econd*{\left\|\nabla_{u_1 ^k} f(x^k) - \nabla_{u_1 ^k} f(w^k) \right\|^2}{\mathcal{F}_k}\\
	={} &\frac{d^2}{|\mathcal{S}_k|}\, \mathbb{E}\Econd*{\left\|\langle \nabla f(x^k), u_1 ^k\rangle u_1 ^k - \langle \nabla f(w^k), u_1^k \rangle u_1^k \right\|^2}{\mathcal{F}_k}\\
	={}&\frac{d^2}{|\mathcal{S}_k|} \mathbb{E}\Econd*{\left(\langle \nabla f(x^k) - \nabla f(w^k), u_1 ^k\rangle \right)^2}{\mathcal{F}_k}.
\end{aligned}
\]
By using $\mathbb{E}\Econd*{u_1 ^k {(u_1 ^k)} ^T}{\mathcal{F}_k} = \frac{1}{d} I_d$, we get
\begin{align}\label{eq:variance2.2}
	& \mathbb{E}\Econd*{\left\|\nabla_{\mathcal{S}_k} f(x^k) - \nabla_{\mathcal{S}_k} f(w^k) + \nabla f (w^k) - \nabla f(x^k)\right\|^2}{\mathcal{F}_k}\nonumber\\
	\leq{} & \frac{d}{|\mathcal{S}_k|}\, \mathbb{E}\Econd*{\left\| \nabla f(x^k) - \nabla f(w^k) \right\|^2}{\mathcal{F}_k}\nonumber\\
	\leq{} &  \frac{2dL}{|\mathcal{S}_k|}(f(w^k) - f(x^k) + \langle \nabla f(x^k), w^k - x^k\rangle).
\end{align}
Plugging \eqref{eq:variance2.2} back into \eqref{eq:variance2.1} and using $\frac{d(d-1)}{|\mathcal{S}_k|}+d^2\leq 2d^2$ concludes the proof.

\subsection{Proof of Lemma 3}
We have
\begin{align*}
	& \frac{M}{2\eta}\left\|z^{k+1}-z^{k}\right\|^{2}+\left\langle g^{k},z^{k+1}-z^{k}\right\rangle  \\
	={} & \frac{1}{\theta}\left(\frac{M}{2\eta\theta}\left\|y^{k+1}-x^k\right\|^2+\left\langle g^k,y^{k+1}-x^k\right\rangle\right)  \\
	={} & \frac{1}{\theta}\left(\frac{L}{2}\left\|y^{k+1}-x^k\right\|^2+\left\langle\nabla f(x^k),y^{k+1}-x^k\right\rangle+\left(\frac{M}{2\eta\theta}-\frac{L}{2}\right)\left\|y^{k+1}-x^k\right\|^2\right)\\
	&+\frac{1}{\theta}\left(\left\langle g^k-\nabla f(x^k),y^{k+1}-x^k\right\rangle\right)\\
	\geq{} & \dfrac{1}{\theta}\left(f(y^{k+1})-f(x^k)+\left(\frac{M}{2\eta\theta}-\frac{L}{2}\right)\left\|y^{k+1}-x^k\right\|^2+\left\langle g^k-\nabla f(x^k),y^{k+1}-x^k\right\rangle\right)\\
	\geq{} & \dfrac{1}{\theta}\left(f(y^{k+1})-f(x^k) - \frac{\eta \theta}{2(M-L\eta\theta)}\|g^k - \nabla f(x^k)\|^2\right).
\end{align*}
Here the first step comes from the update rule for $y_{k+1}$ in \textit{ZO-L-Katyusha}; we used the smoothness of $f$ in the third step; in the last step we used $\langle a, b\rangle \geq -\frac{\|a\|^2}{2\epsilon} -\frac{\|b\|^2 \epsilon}{2}$ for any $\epsilon > 0$.
\subsection{Proof of Lemma 4}

From the algorithm we have
\[
z^{k+1}=\operatorname{prox}_{\frac{\eta}{(1+\eta\sigma)M}\psi}\!\bigg(\frac{\eta\sigma x^{k}+z^{k}-\frac{\eta}{M}g^{k}}{1+\eta\sigma}\bigg),
\]
and so by the first-order optimality condition for minimizing convex functions, we can find $\xi^k\in\partial \psi(z^{k+1})$ such that 
\begin{equation*}
	z^{k+1}-\frac{1}{1+\eta\sigma}\left(\eta\sigma x^k+z^k-\frac{\eta}{M}g^k\right)+\frac{\eta}{(1+\eta\sigma)M}\xi^k=0.
\end{equation*}
Together with $\mu_{f} = M\sigma$, we obtain
\begin{equation*}
	g^k=\dfrac{M}{\eta}(z^k-z^{k+1})+\mu_f(x^k-z^{k+1})-\xi^k.
\end{equation*}
Therefore, we have
\begin{align*}
	\langle g^{k},z^{k+1}-x^{*}\rangle  
 ={} & \mu_{f}\langle x^{k}-z^{k+1},z^{k+1}-x^{*}\rangle+{\frac{M}{\eta}}\langle z^{k}-z^{k+1},z^{k+1}-x^{*}\rangle-\langle \xi^k ,z^{k+1}-x^{*}\rangle  \\
	={} & {\frac{\mu_{f}}{2}}\left(\|x^{k}-x^{*}\|^{2}-\|x^{k}-z^{k+1}\|^{2}-\|z^{k+1}-x^{*}\|^{2}\right) \\
	&+\frac{M}{2\eta}\left(\|z^{k}-x^{*}\|^{2}-\|z^{k}-z^{k+1}\|^{2}-\|z^{k+1}-x^{*}\|^{2}\right)-\langle \xi^k, z^{k+1}-x^{*}\rangle  \\
	\leq{} &\frac{\mu_{f}}{2}\|x^{k}-x^{*}\|^{2}+\frac{M}{2\eta}\left(\|z^{k}-x^{*}\|^{2}-\left(1+\frac{\eta\mu_f}{M}\right)\|z^{k+1}-x^{*}\|^{2}\right) \\
	&-\frac{M}{2\eta}\|z^{k}-z^{k+1}\|^{2}+\psi(x^{*})-\psi(z^{k+1})-\frac{\mu_{\psi}}{2}\|z^{k+1}-x^{*}\|^{2} \\
	={} &\frac{\mu_f}{2}\|x^k-x^*\|^2+\frac{M}{2\eta}\left(\|z^k-x^*\|^2-\left(1+\frac{\eta\mu}{M}\right)\|z^{k+1}-x^*\|^2\right) \\
	&\quad-\frac{M}{2\eta}\|z^{k}-z^{k+1}\|^{2}+\psi(x^{*})-\psi(z^{k+1}),
\end{align*}
where the third step follows from the fact that $\psi$ is $\mu_\psi$-strongly convex, and in the last equality we used $\mu=\mu_{\psi} + \mu_{f}$.

Since $\mathcal{Z}_k=\frac{M+\eta\mu}{2\eta}\|z^{k}-x^{*}\|^{2}$, rearranging terms yields the desired result.

\end{document}